\documentclass{amsart}[18pt]
\usepackage[T1]{fontenc}
\usepackage[left=2cm, right=2cm, top=2cm]{geometry}

\usepackage{amssymb}
\usepackage{url}
\newcommand{\red}{{\text{\rm red}}}

\newcommand{\sA}{\mathcal{A}}
\newcommand{\sL}{\mathcal{L}}
\newcommand{\sO}{\mathcal{O}}
\newcommand{\sG}{\mathcal{G}}
\newcommand{\bL}{\mathbb{L}}
\newcommand{\bN}{\mathbb{N}}
\newcommand{\bZ}{\mathbb{Z}}
\newcommand{\bQ}{\mathbb{Q}}
\newcommand{\bA}{\mathbb{A}}
\newcommand{\fm}{\mathfrak{m}}
\newcommand{\fl}{\mathfrak{l}}
\newcommand{\Spec}{\operatorname{Spec}}
\newcommand\Var[1]{\mathbf{Var}_{#1}}
\newcommand{\into}{\hookrightarrow}
\newcommand{\onto}{\twoheadrightarrow}
\newcommand{\colim}{\varinjlim}
\renewcommand\dim[1]{\mbox{dim}{#1}}
\newcommand\spec[1]{\operatorname{Spec}(#1)}
\newcommand\grot[1]{{\mathbf {Gr}(#1)}}
\newcommand{\ord}{{\operatorname{ord}}}

\newtheorem{theorem}{Theorem}[section]
\newtheorem{lemma}[theorem]{Lemma}

\theoremstyle{definition}
\newtheorem{definition}[theorem]{Definition}
\newtheorem{example}[theorem]{Example}

\newtheorem{proposition}[theorem]{Proposition}
\newtheorem{conjecture}[theorem]{Conjecture}
\newtheorem{assumption}[theorem]{Assumption}
\newtheorem{corollary}[theorem]{Corollary}

\theoremstyle{remark}
\newtheorem{remark}[theorem]{Remark}

\numberwithin{equation}{section}

\begin{document}

\title{The Auto Igusa-Zeta function of a plane curve singularity is rational.}


%    author one information
\author{Andrew R. Stout}
\address{Andrew R. Stout\\
 Dept. of Math. \& Comp. Sci. \\
Bronx Community College (CP 315)\\
2155 University Avenue\\
Bronx, NY 10453}
\email{astout@gradcenter.cuny.edu}

%    \subjclass is required.

\date{}

\dedicatory{}

%    "Communicated by" -- provide editor's name; required.
\commby{Lev A. Borisov}

%    Abstract is required.
\begin{abstract}
We show that the auto Igusa-zeta function $\zeta_{C,p}(t)$ of a plane curve $C$ over an algebraically closed field $k$ is rational away from points $p\in C$ of wild ramification--i.e., it is of the form $f(t)/g(t)$ where $f(t)\in \grot{\Var{k}}[\bL^{-1},t]$, where $\grot{\Var{k}}$ is Grothendieck ring of varieties, and $g(t)=\prod_{i=1}^{n}(1-\bL^{a_i}t^{b_i})$ with $a_i\in\bZ$ and $b_i\in\bN\setminus\{0\}$, where $\bL:=[\bA_{k}^1]$ is the Leftshetz motive. As a consequence, we give a new characterization for a curve $C$ on a smooth surface $S$ to be smooth at a point $p$ on $C$ when the ground field is algebraically closed and of characteristic zero. 
\end{abstract}

\maketitle

\section{Introduction}

The auto Igusa-zeta function, which was originally introduced by Schoutens in \cite{Sch2}, is a motivic generating series associated to a germ $(X,p)$ where $p$ is a point on a variety $X$. A motivic generating series is a formal power series $\sum_{i=0}^n a_n t^n$ where the coefficients $a_n$ are elements of the (localized) Grothendieck ring of varieties $\grot{\Var{k}}[\bL^{-1}]$. In the case of the auto Igusa-zeta function, we form the coefficients $a_n$ in the following way. 

Let $X$ be a separated scheme of finite type over a field $k$. For each $n\geq 1$, we let $J_p^nX:=\Spec(\sO_{X,p}/\fm_p^n)$, and we call this connected, zero-dimensional scheme, the  $n${\it th}-{\it jet of} $X$ {\it at} $p$. Since $J_p^nX$ is finite over the residue field $\kappa(p)$ of $X$ at $p$, the functor $\mbox{Hom}_{\kappa(p)}(J_p^nX\times_{\kappa(p)}-, Y)$, which parameterizes all $J_p^nX$-points of $Y$ exists in the category of $\kappa(p)$-schemes.  We denote this Hilbert scheme by $\nabla_{J_p^nX}Y$, and we call it the {\it (generalized) arc space of} $Y$ {\it along} $J_p^nX$. Moreover, $\nabla_{J_p^nX}Y$ is separated and of finite type over $\kappa(p)$ provided that $Y$ is also--i.e., we may consider its class $[\nabla_{J_p^nX}Y]$ in $\grot{\Var{\kappa(p)}}[\bL^{-1}]$. In the case that $Y$ is itself the $n$-th jet of $X$ at $p$, we obtain the Hilbert scheme $\nabla_{J_p^nX}J_p^nX$, 
which parameterizes all endomorphisms of $J_p^nX$. We call this space the {\it auto-arc space of} $X$ {\it at level} $n$, and we use the notation $\sA_n(X,p)$ to denote $\nabla_{J_p^nX}J_p^nX$ in this paper. The coefficients $a_n$ defining the auto Igusa-zeta function as a motivic generating series are given by $$a_n = [\sA_{n+1}(X,p)]\bL^{-d\cdot\ell(J_p^{n+1}X)},  \ \ \mbox{\ for all } n\geq 0, $$ where $d = \mbox{krull-dim}(\sO_{X,p})$ and where $\ell(J_p^{n+1}X)$ is the length of the Artinian ring $\sO_{X,p}/\fm_p^{n+1}$ (i.e., the dimension of $\sO_{X,p}/\fm_p^{n+1}$ as a vector space over $\kappa(p)$). 

\begin{definition}\label{one}
Let $X$ be a separated scheme of finite type over a field $k$ and let $p$ be a point on $X$ with residue field $\kappa(p)$. The {\it auto Igusa-zeta function} $\zeta_{X,p}(t)$ {\it of the germ} $(X,p)$ is the element of $\grot{\Var{\kappa(p)}}[\bL^{-1}][[t]]$ defined by
\begin{equation}\zeta_{X,p}(t) :=  \sum_{i=0}^{\infty}[\sA_{n+1}(X,p)]\bL^{-d\cdot\ell(J_p^{n+1}X)}t^n ,\end{equation}
where $d = \mbox{krull-dim}(\sO_{X,p})$.
Moreover, if $Y$ is a separated scheme of finite type over $\kappa(p)$ with pure dimension $d$, the {\it (generalized) motivic Igusa-zeta function} $\Theta_{Y,J_p^{\infty}X}(t)$ {\it of} $Y$ {\it with respect to the infinite jet} $J_p^{\infty}X$ is the element of $\grot{\Var{\kappa(p)}}[\bL^{-1}][[t]]$ defined by
\begin{equation}\Theta_{Y,J_p^{\infty}Y}(t):= \sum_{i=0}^{\infty}[\nabla_{J_p^{n+1}X}Y]\bL^{-d\cdot\ell(J_p^{n+1}X)}t^n,\end{equation}
where $J_p^{\infty}X$ denotes the formal scheme formed by the filtered colimit $\varinjlim_n J_p^nX$.
\end{definition}

The original motivation for studying auto-arc spaces $\sA_n(X,p)$ and the auto Igusa-zeta function is to understand the degree to which one may generalize some notions of motivic integration from varieties to schemes. Moreover, it is also of interest to investigate whether or not results in motivic integration have counterparts when we replace {\it linear arc spaces}, defined as $\nabla_{J_p^n\bA_{k}^1}Y$, with more general types of arc spaces such as $\nabla_{J_p^nX}Y$. Given the importance of linear jets and linear arcs space, we let $\fl_n:=J_p^n\bA_{k}^1$ and let $\fl := \colim_n \fl_n$.

\section{Some first results and conjectures.}
In \S 4 of \cite{Sch2}, Schoutens showed that if $X=\bA_{k}^{d}$ where $k$ is an algebraically closed field and $p\in X$ is a closed point, then $\zeta_{X,p}(t) = \bL^{-d}(1-t)^{-1}.$ We reprove his statement here.

\begin{example}\label{linauto}
Consider the case where $X$ is $\bA_{\kappa}^{1}$ and let $p$ be any point of $\bA_{k}^{1}$. 
We will calculate the reduction of the auto-arc space $\sA_n(\bA_{\kappa}^1,p)$. To do this, we let $\alpha:= \sum_{i=0}^{n-1} a_it^i \in \kappa(p)[t]/(t^n)$ with $a_i \in \kappa(p)$ and set $\alpha^n=0$ as an element of $\kappa(p)[t]/(t^n)$. Now, $0=\alpha^n = (a_0 + t\cdot\beta)^n$ where $\beta \in \kappa(p)[t]/(t^n)$, which implies that $a_0^n=0$ and so $a_0= 0$ in the reduction. Therefore, the reduced auto-arc space of $\bA_{k}^{1}$ at $p$ is defined by the equations $a_0=0$ and $(t\beta)^n=0$. However, the second equation is trivially satisfied -- i.e., $(t\beta)^n = t^n\cdot \beta^n= 0\cdot\beta^n = 0$ for any $\beta \in \kappa(p)[t]/(t^n)$. Equivalently, we have the following isomorphism
\begin{equation*}
\sA_n(\bA_{k}^1,p)^{\red}\cong \spec{\kappa(p)[a_0,\ldots,a_{n-1}]/(a_0)}.
\end{equation*}
for all $n\in\bN$. Thus, for all $n\in\bN$, we have 
\begin{equation*}
\sA_n(\bA_{k}^1,p)^{\red}\cong \bA_{\kappa(p)}^{n-1}.
\end{equation*}
\end{example}

\begin{lemma}\label{lem}
Let $p$ be a point of $\bA_{k}^{d}$. Then, for all $n\in\bN$,
\begin{equation*}
\sA_n(\bA_{k}^{d},p)^{\red} \cong \bA_{\kappa(p)}^{r_n}
\end{equation*}
where $r_n=d\cdot(\ell(J_p^n\bA_{k}^{d})-1)$.
\end{lemma}

\begin{proof}
By definition, the coordinate ring of $J_p^n\bA_{k}^{d}$ is isomorphic to
$$\kappa(p)[x_1,\ldots,x_d]/(x_1,\ldots,x_d)^n.$$ Thus, for each $i=1,\ldots,d$, 
we define $\alpha_i=\sum_{|j|<n}a_j^{(i)}x^j$ where $j$ is a multi-index (i.e., 
$j=(j_1,\ldots,j_d)$, $x^j = \prod_{s=1}^{d}x_s^{j_s}$, and 
$|j|=\sum_{s=1}^{d}j_s< n$) and where $a_j\in\kappa(p)$.
Then, the equations defining the auto-arc space are given by 
\begin{equation*}
0=\alpha_{i}^{n} = (a_{0}^{(i)} + \beta_i)^n, \quad \forall i =1,\ldots,d
\end{equation*}
where $0$ is treated as the multi-index $(0,\ldots,0)$ and 
$\beta_i\in(x_1,\ldots,x_d)/(x_1,\ldots,x_d)^n$. This implies 
that $0=(a_{0}^{(i)})^n$ for all $i=1,\ldots,d$ on 
$\sA_n(\bA_{k}^{d},p)$. Thus, in the reduction, 
$0=a_{0}^{(i)}$ for all $i=1,\ldots,d$. Clearly, $\beta^n=0$ 
for all $\beta\in (x_1,\ldots,x_d)/(x_1,\ldots,x_d)^n$. Thus, 
$\sA_n(\bA_{k}^{d},p)^{\red}$ is defined by the equations 
$0=a_{0}^{(i)}$ for all $i=1,\ldots,d$ while $a_{j}^{(i)}$ 
are free variables for $0< |j|<n$. 
Thus, $\sA_n(\bA_{k}^{d},p)^{\red}$ is isomorphic to $\bA_{\kappa(p)}^{r_n}$ for some non-negative integer $r_n$.  It is immediate then that $r_n$ is equal to  $d\cdot(\ell(J_p^n\bA_{k}^{d})-1)$.
\end{proof}

\begin{example}\label{firstex}
Let $X$ be $\bA_{k}^{1}$ and let $p$ be a point of $\bA_{k}^{1}$, then by Example \ref{linauto}, we have
\begin{equation*}
\begin{split}
\zeta_{\bA_{k}^{1}, p}(t) &= \sum_{n=0}^{\infty}[\sA_{n+1}(\bA_{k}^1,p)^{\red}]\bL^{-n-1}t^n\\
&=\sum_{n=0}^{\infty}[\bA_{\kappa(p)}^{n}]\cdot\bL^{-n-1}t^n \\
&= \sum_{n=0}^{\infty}\bL^{n}\cdot\bL^{-n-1}t^n = \bL^{-1}\cdot\sum_{n=0}^{\infty}t^n.
\end{split}
\end{equation*}
Thus, for every point $p$ of $\bA_{k}^{1}$, we have \begin{equation}
\zeta_{\bA_{k}^{1}, p}(t) = \bL^{-1}\cdot\frac{1}{1-t}.
\end{equation}
Thus, $\zeta_{\bA_{k}^{1}, p}(t)$ is an element of $\grot{\Var{\kappa(p)}}[\bL^{-1},t,(1-t)^{-1}]$.
\end{example}

\begin{example} By Lemma \ref{lem}, we may perform an entirely similar calculation as in Example \ref{firstex} to obtain Schoutens' result:
\begin{equation*}
\zeta_{\bA_{k}^{d},p}(t) = \bL^{-d}\cdot \frac{1}{1-t}
\end{equation*}
as an element of $\grot{\Var{\kappa(p)}}[\bL^{-1},t,(1-t)^{-1}]$.
\end{example}

As in \cite{St}, one may generalize Schoutens' result to obtain the following proposition.

\begin{proposition}\label{first}
Let $k$ be an algebraically closed field and
let $f: X \to Y$ be a morphism of separated $k$-schemes which is \'{e}tale at a closed point $p \in X$. Then, 
$$\zeta_{X,p}(t) = \zeta_{Y,f(p)}(t)  $$
as elements of $\grot{\Var{k}}[\bL^{-1}][[t]]$.
\end{proposition}

\begin{proof}
Since $f$ is \'{e}tale, $\kappa(p)$  is a separable field extension of $\kappa(f(p))$. Since $k$ is algebraically closed $k\cong\kappa(p)\cong \kappa(f(p))$, and the canonical ring homomorphism
\begin{equation*}
e:\hat\sO_{Y,f(p)}\to \hat\sO_{X,p}
\end{equation*}
is an isomorphism, cf., Problem 10.4 of Chapter III, \S 10 of \cite{Ha1}. Thus, for each $n\in \bN$, we have that $e$ induces a morphism
\begin{equation*}
  e_n : J_{p}^{n}X \to J_{f(p)}^{n}Y
\end{equation*}
and, moreover, this is an isomorphism schemes. 
\end{proof}

In particular, if $X$ is smooth at $p$, then
$$\zeta_{X,p}(t) = \bL^{-\dim{}_{p}(X)}(1-t)^{-1} \in \grot{\Var{\kappa(p)}}[\bL^{-1}][[t]] , $$
where $\dim{}_{p}(X):=\mbox{krull-dim}(\sO_{X,p})$.
 It is theorized further in \cite{St}  that the auto Igusa-zeta function $\zeta_{X,p}(t)$ is a perfect local invariant -- i.e., we propose the following conjecture 

\begin{conjecture} \label{conj1} Let $k$ be an algebraically closed field and let $X$ and $Y$ be separated $k$-schemes of finite type.  If there are closed points $p\in X$ and $q\in Y$ such that $\zeta_{X,p}(t) = \zeta_{Y,q}(t)$, then $(X,p)$ is analytically isomorphic to $(Y,q)$. 
\end{conjecture}

In particular, a weak form of this conjecture will be that if $$\zeta_{X,p}(t) = \bL^{-\dim{}_{p}(X)}(1-t)^{-1}, $$ then $X$ is smooth at $p$. We show in \S \ref{proof} that this conjecture is true when $X$ is a plane curve. 

In \cite{St}, I explicitly computed the auto Igusa-zeta function in the following three singular cases: 1) $X$ is the node $xy=0$, 2) $X$ is the nodal cubic $y^2=x^3-x^2$, and 3) $X$ is the cuspidal cubic $y^2=x^3$. Furthermore, I indicated how to do such computations for more complicated singular points of algebraic curves. In summary, there is computational evidence that: 
\begin{itemize}
\item  Conjecture \ref{conj1} will be true.
\item $[X]\zeta_{X,p}(t)=\Theta_{X,\fl}(t)$ iff $X$ is a smooth.
\item $\zeta_{X,p}(t)$ is a rational function.
\item Every pole of $\zeta_{X,p}(t)$ is also a pole of $\Theta_{X,\fl}(t)$.
\end{itemize}

Indeed, we will show in this paper that each of these bullets will hold when $X$ is a curve on a smooth surface over an algebraically closed field of characteristic zero. For the second bullet, we need to assume $X$ has only one singular point as well.
 
 \section{Auto-arc spaces as fibers of generalized arc spaces}

As we will need it in the rest of the paper, we will now describe abstractly how to calculate the generalized arc space $\nabla_{J_pC}X$ whenever $X$ is affine. The following description occurs almost verbatim in \S 4 of \cite{Sch2}. Let $k[x_1,\ldots,x_n]/I$ with $(f_1,\ldots,f_m)$ be the coordinate ring of $X$, and let $R$ be the coordinate ring of $J_p^nC$. We may choose a basis $\beta =\{b_1,\ldots,b_\ell\}$, where $\ell$ is the length of $J_p^nC$, such that $b_1=1$ and such that we obtain a Jordan-Holder composition series
$$0 \subsetneq \beta_1 \subsetneq \beta_2 \subsetneq \cdots \subsetneq \beta_{\ell} \subsetneq R$$
where $\beta_i := \{b_i,\ldots,b_{\ell}\}\cdot R$ as in \S 2.1 of \cite{Sch1}. Therefore, the first $m$ basis elements form a basis for $R/\beta_{m+1}$.  We substitute $\bar x_s:=\sum_{m}x_m^{(s)}b_m$ for each variable $x_s$ in $f_i(x_1,\ldots,x_n)$ to obtain $0=f_i(\bar x_1,\ldots,\bar x_n) = \sum_{j}(\nabla_jf_i)b_j$ where $\nabla_jf_i$ is a polynomial in the variables $x_m^{(s)}$. Then, the coordinate ring of $\nabla_{J_pC}X$ is given by $$k[x_1^{(1)},\ldots,x_{\ell}^{(1)},x_{1}^{(2)},\ldots,x_{\ell}^{(n)}]/(\nabla_1f_1,\ldots, \nabla_{\ell}f_1,\nabla_{1}f_2,\ldots,\nabla_{\ell}f_m).$$

\begin{lemma}\label{lem1}
Let $C$ be any variety, $p$ a point of $C$, and $\kappa(p)$ the residue field of $C$ at $p$. 
Let $U$ be the smooth locus of a $\kappa(p)$-variety $X$. Then, the natural morphism $\rho_n: \nabla_{J_p^nC}X \to X$ is a piecewise trivial fibration over $U$. 
\end{lemma}

\begin{proof}
This reduces to verifying Theorem 4.4 of \cite{Sch2}, which states that $\rho_n^{-1}(U) \cong \nabla_{J_p^nC}U$. It follows from this because if $U$ is smooth, then there is an \'{e}tale morphisms $U \to \bA_{\kappa(p)}^{d}$ where $d =\mbox{krull-dim}(\sO_{X,p})$.  It is straightforward to show $$\nabla_{J_p^nC}U \cong U \times_{\kappa(p)} \nabla_{J_{p}^nC}\mathbb{A}_{\kappa(p)}^{d}\cong U\times_{\kappa(p)}\bA_{\kappa(p)}^m$$ where  $m = d(\ell(J_p^nC)-1)$ as the first isomorphism follows from the infinitesimal lifting property for formally \'{e}tale morphisms, and the second isomorphism is a simple computation much like Example \ref{linauto}. 

To verify Schoutens' claim, we may reduce to the case where $X=\Spec{B}$ with $B=\kappa(p)[x_1,\ldots,x_n]/I$ and where $U$ is a distinguished open set $\Spec{B_f}$ with $B_f=\kappa(p)[y]/(g)$ where $g:=fy-1$. The following verification occurs almost verbatim in the proof of Theorem 4.4 of \cite{Sch2}. If $A$ is the coordinate ring of $\nabla_{J_p^nC}X$, then it follows that the coordinate ring of $\nabla_{J_p^nC}U$ is given by $A[y_1,\ldots,y_{\ell}]/(\nabla_1 g,\nabla_2 g,\ldots, \nabla_{\ell}g)$ where $\ell$ is the length of $J_p^nC$ and where  $b_i$ are the basis elements of the coordinate ring of $J_p^nC$, which implies $\nabla_j g = 0$ for all $j=1,\ldots,\ell.$ Since we choose $b_1$ to be equal to $1$, it is easy to see that $\nabla_1g = \nabla_1fy_1 -1$.

As before, we may adjust our basis further to obtain a Jorden-Holder composition series from which it follows that $\nabla_mg\in A[y_1,\ldots y_m]$ -- i.e., we have that $\nabla_mg =  \nabla_1fy_m + (\mbox{ terms only involving } y_1,\ldots,y_{m-1}).$ By using induction, we have that $\nabla_m g$ is an element of $A[y_1]/(\nabla_1g)$ for all $m$ since $\nabla_1f$ is invertible. Therefore, the coordinate ring of $\nabla_{J_{p}^nC}U$ is of the form $A_{\nabla_1f}$, which is also the coordinate ring of $\rho_n^{-1}(U)$.
\end{proof}

\begin{lemma}\label{lem2}
Let $C$ be a be a variety over an algebraically closed field $k$ and let  $p$ be a closed point of $C$. Then, 
$$\rho_n^{-1}(p)^{\red} \cong \sA_n(C,p)^{\red}, $$
where $\rho_n: \nabla_{J_p^nC}C \to C$.
\end{lemma}

\begin{proof}
Since $J_p^nC$ is a closed subscheme of $C$, it follows that $\sA_n(C,p)$ is a closed subscheme of $\nabla_{J_p^nC}C$. 
Now, without loss of generality, we may assume $C$ is affine with coordinate ring $k[x_1,\ldots,x_e]/I$, which implies that the coordinate ring of $J_p^nC$ is $k[x_1,\ldots,x_e]/(I+\fm_p^n).$ We may also assume that $p$ is given by $(x_1,\ldots,x_e)$, which implies that $ 0= (\bar x_s)^n$ where $\bar x_s = \sum_m x_m^{(s)}b_m$, where $\{b_1,\ldots,b_{\ell}\}$ is a basis for $J_p^nC$. Since $b_1=1$, this immediately implies $(x_1^{(s)})^n = 0$ in the coordinate ring $A$ of $\sA_n(X,p)$. Therefore, the ideal $J$ which defines $A^{\red}$ as a quotient of the coordinate ring $B$ of $\nabla_{J_p^nC}C$ contains $\fm_p\cdot B$. However, $\fm_p\cdot B$ is the ideal defining $\rho_{n}^{-1}(p)$, which implies that $\sA_n(X,p)^{\red}$ is a subscheme of $\rho_n^{-1}(p)^{\red}$. In fact, these two schemes are isomorphic since $x_1^{(s)}=0$ implies that any equation of the form $H(\bar x_1,\ldots,\bar x_e) = 0$ for  $H \in \fm_p^n$ is actually a trivial equation since the basis elements $b_2,\ldots, b_r$ must be nilpotents with nilpotency less than or equal to $n$.
\end{proof}

\begin{lemma} \label{singlemma} Let $k$ be an algebraically closed field and let $C$ be a $k$-variety with only one singular point $p\in C$. Then, 
$$[\nabla_{J_p^nC}C] = [C\setminus\{p\}]\bL^{m}+[\sA_n(C,p)]$$
in $\grot{\mathbf{Var}_k}$ where $m = \dim{}_p(X)(\ell(J_p^nC)-1)$.
\end{lemma}

\begin{proof}
This is a direct consequence of Lemma \ref{lem1} and Lemma \ref{lem2}.
\end{proof}

\begin{theorem}\label{deco} Let $k$ be an algebraically closed field and let $C$ be a $k$-variety with only one singular point $p$. In this case, the motivic Igusa-zeta function of $C$ with respect to $J_p^{\infty}C$ decomposes as $$\Theta_{C,J_p^{\infty}C}(t) = \bL^{-\dim{}_p(C)}\frac{[C]-1}{1-t} + \zeta_{C,p}(t)$$
in $\grot{\Var{k}}[\bL^{-1}][[t]]$.
\end{theorem}

\begin{proof}
This follows from Definition \ref{one} and Lemma \ref{singlemma}.
\end{proof}

\begin{example}\label{cusp}
Let $C$ be the cuspidal cubic defined by $y^2=x^3$ over over an algebraically closed field $k$. This has one singularity at the origin $O=(0,0)$.
We immediately arrive at
\begin{equation*}
\Theta_{C,J_O^{\infty}C}(t) = \bL^{-1}\frac{[\mathbb{G}_m]}{1-t} + \zeta_{C,O}(t).
\end{equation*}
I showed in \S 7 of \cite{St} that $\zeta_{C,O}(t)$ is rational provided $\mbox{char}(k)\neq 2,3$ -- in fact, under these conditions, we arrived at the following explicit formula:
\begin{equation*}
\zeta_{C,O}(t)= \bL^{-1}+\bL t+\bL^2t^2+ \frac{(\bL^7-\bL^6)t^3+\bL^7t^4+\bL^7t^{7}}{(1-\bL t^3)(1-t)}.
\end{equation*}
\end{example}

\section{Rationality of the auto Igusa-zeta function of a plane curve.}

In \cite{St}, we used a computational approach to find explicit formulas for the auto Igusa-zeta function of a few algebraic curves and arrived at a basic idea for how the auto Igusa-zeta function behaves for curves. Namely, it appears that finding the expansion of $\zeta_{C,p}(t)$ in terms of the classical motivic zeta function is closely related to its normalization. In this section, we will gain a better picture of why exactly that must be the case.

\begin{assumption} \label{ap}In this section, we work relative to the following set-up.
We let $C$ be a curve on a smooth surface over an algebraically closed field $k$. We let $\gamma: \bar C \to C$ be the normalization of $C$. We fix a point $p$ of $C$, and we assume that the characteristic of $k$ does not divide $\mbox{mult}_{q}(\bar C)$ for all $q \in \gamma^{-1}(p)$. Note that by a curve we mean that $C$ is a reduced separated $k$-scheme of finite type which has pure dimension $1$. Also, note that the last condition means that $\mbox{char}(k)$ will not divide $\mbox{mult}_p(C)$.
\end{assumption}

\begin{remark}
If we start with a smooth point $p$ of $C$, then $\gamma: \bar C \to C$ will be an isomorphism away from the exceptional locus $\gamma^{-1}(C_{\mbox{sing}})$. In particular, $\hat\sO_{C,p} \cong \hat\sO_{\bar C, q}\cong k[[t]]$. This immediately implies that $J_p^{\infty}C \cong \mbox{Spf}(k[[t]])$ and for any $k$-scheme $X$, $\nabla_{\fl}X = \varprojlim_n\nabla_{J_p^nC}X$ is the usual arc space. Note that 
$\nabla_{\fl}X$ is most commonly denoted by $\sL(X)$  in the literature (cf. \cite{DL1}).  In this paper, we are mostly concerned with the case when $p$ is a singular point of $C$.
\end{remark}

Before we begin the main theorems, we need to introduce a few concepts from motivic integration. A subset $S$ of $\nabla_{\fl}X$ is said to be a {\it cylinder} if for some $m$, $S = \pi_m^{-1}(A)$ for some constructible subset $A$ of $\nabla_{\fl_m}X$ where $\pi_m : \nabla_{\fl}X\to \nabla_{\fl_m}X$ is the canonical truncation morphism induced by the canonical homomorphism $k[[t]] \to k[t]/(t^{m})$. Moreover, $S$ is said to be a {\it closed cylinder} if $A$ is a closed subset of $\nabla_{\fl_m}X$. Every cylinder is {\it strongly measurable} under the motivic measure $\mu$ on $\nabla_{\fl}X$ (for reference, see \S2.3 of \cite{MB} and also  the remark immediately after Theorem A.6 of \cite{DL0}). This last fact implies that, whenever $k$ admits a resolution of singularities (e.g., $k$ is of characteristic zero), the motivic change of variables formula holds for cylinders (cf., Theorem A.10 \cite{DL0}).

Given a cylinder $S$, we have the motivic Poincar\'{e} series $P_{S}(t)$ of $S$, which is defined by 
\begin{equation}
P_S(t):=\sum_{n=0}^{\infty}[\pi_{n+1}(S)]t^n \in \grot{\Var{k}}[\bL^{-1}][[t]]. 
\end{equation}
Since $S$ is cylinder, we may apply the motivic change of variables formula to obtain Theorem 5.4 of \cite{DL2}, which states that $P_S(t)$ is a {\it rational function}. We reserve the notion of {\it rational function} to mean that $P_S(t) = f(t)/g(t)$ where $f(t)$ is a polynomial in  $\sG[t]$, where the ring $\sG$ denotes the completion of $\grot{\Var{k}}[\bL^{-1}]$ along the dimensional filtration\footnote{The dimensional filtration, introduced by Kontsevich in \cite{K}, is given by subgroups of the form $F^m:= \{S\bL^{-i}\mid \dim(S)-i \leq m\}$.}, and where $g(t) = (1-\bL^{a_1}t^{b_1})(1-\bL^{a_2}t^{b_2})\cdots(1-\bL^{a_k}t^{b_k})$ with $a_i \in\bZ$ and $b_i\in\bN\setminus\{0\}$ for all $i=1,\ldots,k$.
Following the proof of Theorem 5.1' of \cite{DL2}, we have that $P_S(t)$ is a finite $\sG$-linear combination of series of the form $f(\bL^{-1},t)$ where $f(x,y)\in\bZ[x][[y]]$, which by Lemma 5.2 of \cite{DL2} gives the claim that $P_S(t)$ is rational.
Clearly, this notion of rationality is stronger than just requiring a motivic generating series to be an element of $\sG(t)$. The former notion is preferable not just because it is stronger but also because it displays the types of {\it (candidate) poles} which  may occur in a rational motivic generating function. 

\begin{remark}
By using higher order angular component maps $\mbox{ac}_n$ to the residue field, one may indeed replace $\sG$ with $\grot{\Var{k}}[\bL^{-1}]$ in the results of this section as cylinders will then be semi-algebraic. This follows from the work of Cluckers and Loeser in \cite{CL1} for the characteristic zero case (and \cite{CL2} for the mixed characteristic case). The result is achieved by performing all necessary constructions within the world of definable sets and replacing limits by sums of geometric series, cf. \cite{L2}.
\end{remark}

\begin{lemma}\label{rzeta}
Let $S$ be a cylinder of $\nabla_{\fl}X$. For each $r\geq 1$ and $e\in\bN$, we define  
\begin{equation}
Z_S^{r,e}(t) := \sum_{n=0}^{\infty}[\pi_{r(n+1)+e}(S)]\bL^{-r(n+1)-e}t^n . 
\end{equation}
This is the motivic generating function $\sum_{i=0}^{\infty}a_it^i$ whose coefficients $a_i$ are the $(r\cdot i+e)$-th coefficients of the motivic generating function $\bL^{-r}P_S(\bL^{-1}t)$. In particular, $Z_S^{r,e}(t)$ is a rational function. 
\end{lemma}

\begin{proof}
The first claim is immediate.  For the second claim, note that $\equiv_r$ is a Presburger condition, and so $\pi_{r(n+1)+e}^{-1}(\pi_{r(n+1)+e}(S))$ forms a semi-algebraic family of strongly measurable sets (here the family is given by $i\equiv_re$) from which the claim follows from Theorem 5.1', Lemma 5.2, and Theorem 5.4 of \cite{DL2} as outlined above in the remarks concerning the rationality of $P_S(t)$.
\end{proof}

\subsection{Preliminaries on the normalization of a curve.}

Again, let $(C,p)$ be a pointed curve subject to Assumption \ref{ap} and consider its normalization $\gamma: \bar C\to C$. For each branch $q \in \gamma^{-1}(p)$, the Newton-Puiseux Theorem\footnote{This is where our assumptions that $k$ is algebraically closed and that $\mbox{char}(k)$ does not divide $\mbox{mult}_{q}(\bar C)$ for all $q \in \gamma^{-1}(p)$. } states that given local coordinates $(x,y)$ of the point $q$, we have
\begin{equation*}
\begin{split}
x &= t^r \\
y &= \sum_{i=1}^{\infty} a_{i}t^i
\end{split}
\end{equation*}
with $a_{i}\in k$. By 1.89 and Theorem 1.96 of \cite{JK}, we have that $r=\mbox{mult}_{q}\bar C$. In other words, the natural number $r$ is the multiplicity of $\bar C$ at the branch $q$. 

In fact, if we let $f(x,y)\in k[x,y]$ be such that $f(x,y) = 0$ defines $C$, then the image $F(x,y)$ of $f(x,y)$ in  $k[[x,y]]$ factors as
$F(x,y) = u(x,y) \prod_{j=1}^{s}G_j(x,y)$ where $u(x,y)$ is a unit and the $G_j(x,y)$ are distinct irreducible power series. The power series $G_j(x,y)$ are sometimes also referred to as the branches of $C$. In fact, using the Newton-Puiseux Theorem, there is a ring homomorphism $\phi_j : k[[x,y]] \to k[[t]]$ such that
$$\phi_j(G_j(x,y)) = G_j(t^{r_j},\sum_{i=1}^{\infty} a_{i,j}t^i)=0  $$
which therefore induces an embedding $k[[x,y]]/(G_j(x,y))\into k[[t]]$. 

It then follows that if we let $\hat\sO_{\bar C, q}\cong k[[t]]$ be the completion of the local ring $\sO_{\bar C, q}$ with $q\in\gamma^{-1}(p)$ 
and let $\hat\sO_{C,p}$ be the completion of the local ring $\sO_{C,p}$, which is isomorphic to $k[[x,y]]/(F(x,y))$,  then there is an embedding of rings $$\hat \sO_{C,p} \into \hat\sO_{\bar C, q}$$ making $\hat \sO_{\bar C, q}$ a finite $\hat\sO_{C,p}$-algebra.

In particular, if $C$ is unibranched (meaning that $s=1$), then 
 $\hat\sO_{C, p}$ embeds into $k[[t]]$ as a subring and this ring embedding is induced by $\phi:=\phi_1$.  Moreover, if $s\geq 1$, we have an embedding 
 $\phi : \hat\sO_{C, p} \into \oplus_{j=1}^{s} k[[t]]$ where $\phi := \oplus_{j=1}^{s}\phi_j$. On the geometric side, $\gamma_n :  E_n \to \spec{\sO_{C,p}/\fm_p^n}$ where $E_n = \bar C\times_k  \spec{\sO_{C,p}/\fm_p^n}$. Taking filtered colimits in formal schemes, which is the same as completing the integral closure of $\sO_{C,p}$ with respect to the Jacobson radical, gives a formal scheme $$\hat E \cong \sqcup_{j=1}^{s} \mbox{Spf}(k[[t]])$$ and a surjective morphism $\hat\gamma : \hat E \to \mbox{Spf}(\hat\sO_{C,p})$ such that $\hat \gamma = \mbox{Spf}(\phi)$. The morphism $\hat\gamma$ will be called the {\it uniformization morphism of the analytic germ} $(C,p)$.
 
\begin{example} \label{examuni}
Let $r> 1$ and $m> r$ with $m$ and $r$ coprime. Consider the curve defined by $f(x,y) = y^r - x^m = 0$, which has a singular point at the origin $O$. This is a unibranched curve, and the homomorphism $\phi_1 : k[[x,y]] \to k[[t]]$ defined by $\phi_1(x) = t^r$ and $\phi_1(y) = t^m$  induces an embedding $\hat\sO_{C,O} \into k[[t]]$. Moreover, one can show that $\gamma : \bar C \to C$ is a homeomorphism of the underlying topological spaces as it is given by $t \mapsto (t^r,t^m)$.
\end{example}

\subsection{Case of a unibranched curve.}

In this subsection, we focus entirely on the unibranched case as outlined above -- i.e., we let $C$ be a plane curve with only one singular point, which we may assume without loss of generality occurs at the origin $O$. Then, as outlined above, $\hat \sO_{C,O} \cong k[[x,y]]/F(x,y)$ embeds into $\hat\sO_{\bar C, \gamma^{-1}(O)}\cong k[[t]]$ via the homomorphism $\phi : k[[x,y]]/F(x,y) \to k[[t]]$ defined $\phi(x) = t^{r}$ and $\phi(y) = \sum_{i=1}^{\infty}a_it^i$.

Let $\Phi_n$ be the composition of $\phi$ and the canonical map 
$k[[t]]\to k[t]/(t^{rn+e})$ where $e=(r-1)(m-r+1)$ where $m=\mbox{ord}(\phi(y))$. When $m=r$, we take $e=0$. Note that $m\geq r$. This follows from the set-up in the preceding paragraph and Corollary 2.2.8 by \cite{EA}, which states\footnote{Note that $\mbox{mult}_O(C)$ is denoted in \cite{EA} by $e(\gamma)$ where $e$ corresponds to multiplicity and $\gamma$ the algebroid curved defined by the equation $G_j=0$ for some irreducible branch $G_j\in k[x,y]$ of $C$.} $\mbox{mult}_O(C) = \mbox{min}\{r,m\}$, and Theorem 1.96 of \cite{JK} which states $r=\mbox{mult}_O(C)$. Let $R_n$ be the quotient of $k[[x,y]]/(F(x,y)+(x,y)^n)$ with the kernel of $\Phi_n$. We then have an induced injective ring homomorphism $\phi_n : R_n \into k[t]/(t^{rn+e})$. Furthermore, there is a surjective homomorphism
$\psi_n: k[[x,y]]/(F(x,y)+(x,y)^n) \onto R_n $ determined by choosing a basis $\beta=\{b_1,\ldots,b_{\ell}\}$ for $k[[x,y]]/(F(x,y)+(x,y)^n)$ and sending a certain requisite number of basis elements to zero. Moreover, we may form a Jorden-Holder composition series, as in the beginning of \S 3, so that $k[[x,y]]/(F(x,y)+(x,y)^n) \mbox{ mod } \beta_{q-1}$ is isomorphic to $R_n$ with $\beta_{q-1} = \{b_{\ell-q+1},\ldots,b_{\ell}\}$ where $\ell$ is the length of $R_n$ and $q$ is the length of $k[[x,y]]/(F(x,y)+(x,y)^n)$.  
More explicitly, we let $\beta=\{x^ay^b\mid ar+bm<rn+e, \ b< r\}$. 
We choose the monomial ordering so that $\beta_{q-1}$ can be described as $$\{x^n,x^{n+1},\ldots,x^{n+\lceil\frac{e}{r}\rceil-1},yx^{n-1}, yx^n,\ldots,yx^{n+\lceil\frac{{e}}{r}\rceil-2},y^2x^{n-2}, \ldots,y^{j-1}x^{n+\lceil\frac{e}{r}\rceil-j}\},$$
where $j = \lceil\frac{r}{m-r}\rceil$. Here, we are assuming $r\neq m$ since $\psi_n$ is an isomorphism in this case (with $e=0$), and we assume $n\geq j-\lceil \frac{e}{r}\rceil$ for simplicity.  We call
$\spec{R_n}$ the {\it auxiliary fat point of the closed germ} $(C,p)$ {\it at level} $n$. 

\begin{lemma}\label{vimp}
Let $(C,p)$ be a pointed unibranched plane curve satisfying Assumption \ref{ap}. Let $F_n=\spec{R_n}$ be the auxiliary fat point of $(C,p)$ at level $n$. Then, the truncation homomorphism $\psi_n: \sO_{C,p}/\fm_p^n \onto R_n$ induces a canonical morphism $(\nabla_{J_p^{\infty}C}C)^{\red} \to (\nabla_{F_n}C)^{\red}$ which is a projection over the singular locus. 
\end{lemma}

\begin{proof}
Without loss of generality, we may assume that $p$ is the origin. We need to show that the arc variables which are coefficients of elements of $\beta_{q-1}$ described above are free over the singular point $p=(0,0)$. Using Weierstrass preparation lemma (cf. 1.89 of \cite{JK}), we have that $C$ is given by the equation 
$$y^r = g_{r-1}(x)y^{r-1}+\ldots+g_1(x)y+g_0(x).$$
By Netwon-Puisseux's Theorem, $x=t^r$, $y=\sum_{i=m}^{\infty}a_it^i$ is a solution, which implies $\mbox{ord}_x(g_i(x)) \geq m - r$ for all $i=0,\ldots, r-1$. Now, we may plug in arcs $\bar y := y_0+y_1t^r +\mbox{ HOT}$ and $\bar x := x_0 +x_1t^r +\mbox{ HOT}$, where HOT stands for higher order terms in $t$, and notice that, over the singular locus $(x_0,y_0) = (0,0)$,  $f(\bar x,\bar y)=0$ is of the form
\begin{equation}\label{hot}
(y_1t^r +\mbox{ HOT})^r = \sum_{i=0}^{r-1}g_i(x_1t^r +\mbox{ HOT})(y_1t^r +\mbox{ HOT})^{i}.
\end{equation}
 Thus, the left hand side has order greater than or equal to $r^2$ and the right hand side has order greater than or equal to $r(m-r)$. There are then two cases: $r \geq m-r$ and $r< m-r$. When $r \geq m-r$, there is nothing to prove as $nr+r \geq nr+r\lceil \frac{e}{r}\rceil$. Thus, the arc variables which are coefficients of elements of $\beta_{q-1}$ are free. 

Now, when $r< m-r$, then the left hand side implies that $y_i$ are in fact nilpotent (and, hence, equal to zero in the reduction) for $i=1,2,\ldots,m-r-1$. Thus, in this case the left hand side has order greater than or equal to $r(m-r)$, and so we reduce to the first case proven above.
\end{proof}

Thus, given a unibranched plane curve with only one singular point $p$ satisfying assumption \ref{ap}, we have that the induced map $(\nabla_{J_p^{\infty}C}C)^{\red} \to (\nabla_{F_n}C)^{\red}$ is a piecewise trivial fibration away from $p$ with fiber $\bA_k^{v}$ where $v=\ell(F_n)-\ell(J_p^nX)$, and, by Lemma \ref{vimp}, it is also a piecewise trivial fibration with fiber $\bA_k^{v}$. In other words, in $\grot{\Var{k}}$, we have 
\begin{equation}\label{v}
[\nabla_{J_p^{\infty}C}C] = [\nabla_{F_n}C]\bL^v.
\end{equation}
We note that $v = \lceil \frac{m-r+1}{r}\rceil\cdot\lceil \frac{r}{m-r}\rceil$ if $m\neq r$ provided that $n\geq \lceil\frac{r}{m-r}\rceil- \lceil\frac{m-r+1}{r}\rceil$ and $v=0$ when $r=m$. We will need this fact in what follows.

\begin{example}
Consider the curve $C$ given by the equation $y^3=x^5$. As noted in Example \ref{examuni}, it is unibranched. Furthermore, in the above, $e=6$, $\frac{e}{r}=2$, and  $j=2$. Thus, for each $n\in\bN$, $\beta_{q-1}= \{x^n,x^{n+1},yx^{n-1},yx^n\}$. Indeed, the kernel of $\phi_n$ is generated by all monomials $x^ay^b$ such that $3a+5b \geq 3n+6$ with $b\leq 2$. The reader my check that the only monomials which are of degree greater than or equal to $n$ which are not sent to zero under $\phi_n$ are the elements of $\beta_{q-1}$. 
Then, we have $5-3 < 3$ and also Equation \ref{hot} is of the form
$$(y_1t^3 +\ldots+y_nt^3n+ \dots)^3 = (x_1t^3 +\ldots+x_nt^3n+ \dots)^5$$
which shows that the lowest term on the left hand side involving $y_n$ is $3y_1^2y_nt^{3n+6}$ which is zero since $t^{3n+6}=0$. The same is obviously true for the right hand side, which shows that $y_n$ is free. The reader may quickly check that the remaining three elements of $\beta_{q-1}$ are free in the same fashion. 
\end{example}

\begin{theorem}\label{unitheorem}
Let $(C,p)$ be a pointed unibranched curve satisfying Assumption \ref{ap}. Let $X$ be a $k$-variety, $r=\mbox{mult}_p(C)$, and let $F_n$ denote the auxiliary fat point of $(C,p)$ at level $n$. Then, there exists a natural number $e$ such that  
\begin{itemize}
\item[(1)] The generalized arc space $\nabla_{F_n}X$ is a closed subvariety of $\nabla_{\fl_{rn+e}}X$.
\item[(2)] The natural truncation morphism $\pi_{n-1}^{n}: \nabla_{F_n}X\to \nabla_{F_{n-1}}X$ is induced by restricting the natural truncation morphism $\nabla_{\fl_{rn+e}}X\to \nabla_{\fl_{r(n-1)+e}}X$.
\item[(3)] The projective limit $\nabla_{J_p^{\infty}C}X\cong\varprojlim_n \nabla_{F_n}X$ is a closed cylinder of $\nabla_{\fl}X$.
\item[(4)] The motivic generating function $Z_S^{r,e}(t)$ defined in Lemma \ref{rzeta} with $S:=\nabla_{J_p^{\infty}C}X$ is well-defined.
\item[(5)] If $X$ has a resolution of singularities (e.g., $\mbox{char}(k) = 0$), then $Z_S^{r,e}(t)$ is rational.
\end{itemize}
\end{theorem}

\begin{proof}
 Let $R_n$ be the coordinate ring of the auxiliary fat point of $(C,p)$ at level $n$ as defined immediately before Theorem \ref{unitheorem}. There is a ring embedding $\phi_n: R_n \into k[t]/(t^{rn+e})$ for all $n\geq 0$ for some fixed $e\in\bN$. For each $n$, we define the set
 \begin{equation}
 \Gamma_n := \{\mbox{ord}_t\phi(a)\mid a \in R_n\setminus \{0\}\} \subset \bZ/(rn+e).
 \end{equation}
 We let $\Sigma_n := \bZ/(rn+e)\setminus \Gamma_n$. 
 From this, we may establish (1) in the following way. Let $\beta$ be a basis for $R_n$ and let $\alpha$ be a basis $k[t]/(t^{rn+e})$. Without loss of generality, we may assume $\alpha =\{a_1\ldots, a_{rn}\}$ with $a_i = t^i$ and that $\beta= \{a_{\gamma}\mid \gamma \in \Gamma_n\}$. These are both given by the obvious monomial ordering. On any open affine $U$ of $X$ with coordinate ring $k[x_1,\ldots,x_v]/(f_1,\ldots,f_w)$, we have $$\nabla_j^{\alpha}f_i\in k[x_1^{(1)},\ldots,x_{rn+e}^{(1)},x_{1}^{(2)},\ldots,x_{rn+e}^{(v)}].$$ Therefore, $\nabla_{F_n}X$  is the closed subscheme of $\nabla_{\fl}X$ defined by the ideal $$I:=(x_{\sigma}^{(i)} \mid \sigma \in \Sigma_n),$$ provided we can show $\nabla_{\gamma}^{\alpha}f_i = \nabla_{\gamma}^{\beta}f_i \mbox{ modulo } I$ for each $\gamma \in \Gamma_n$ and that $\nabla_{\sigma}^{\alpha}f_i = 0 \mbox{ modulo } I$ for each $\sigma \in \Sigma_n$ 
Note that here $\nabla_j^{\alpha}f_i$ and $\nabla_j^{\beta}f_i$ are written to signify their dependence on $\alpha$ and $\beta$, respectively. The first claim is immediate. To show the second claim -- i.e., that $\nabla_{\sigma}^{\alpha}f_i$ is sent to zero by sending the arc variables $x_{\sigma}^{(i)}$ to zero whenever $\sigma\in\Sigma_n$ for $i=1,\ldots, v$ -- we first notice that $\nabla_{\sigma}^{\alpha}f_i$ is the coefficient of $t^{\sigma}$. Thus, given a decomposition
\begin{equation}\label{summy}
\sigma = n_1 +\cdots+n_{k}, \ \ \ n_i \geq 0,
\end{equation}
we note that if all $n_i$ are in $\Gamma_n$, then $\sigma \in \Gamma_n$, which is a contradiction. Thus, there must be at least one $j$ so that $n_j\in\Sigma_n$. Each decomposition given by Equation \ref{summy}, represents a potential term of $\nabla_{\sigma}^{\alpha}f_i$. Putting this together, we have that each term of $\nabla_{\gamma}^{\alpha}f_i$ involves some $x_{n_j}^{(u)}$ as a factor with $n_j\in \Sigma_n$ whenever $\gamma\in \Sigma_n$. Therefore, $\nabla_{\sigma}^{\alpha}f_i = 0 \mbox{ modulo } I,$ which establishes $(1)$. Part $(2)$ is an immediate consequence of $(1)$.
 
To prove $(3)$, it is enough to work with $k$ points. We see that the image of the ring embedding
 $\phi:\hat\sO_{C,p} \into k[[t]]$ can be described set-theoretically as the collection of arcs $a_0+\sum_{i\geq r} a_i t^i$ generated as a $k$-algebra by $1$, $\phi(x)$, and $\phi(y)$, with $\ord_t(\phi(x))=r$ and $\ord_t(\phi(y))=m$. Let $c$ be the conductor--i.e., the minimal natural number so that $(t^c)k[[t]] \subset \hat\sO_{C,p}$. It follows that for any $n\in \bN$ such that  $nr+e > c$, 
 $$\nabla_{J_p^{\infty}C}X = \pi_{nr+e}^{-1}(\nabla_{F_n}X) . $$
Parts $(4)$ and $(5)$ are immediate consequences of part $(3)$, Lemma \ref{rzeta}, and the fact that cylinders are strongly measurable.
\end{proof}

\subsection{The multibranched case.}

Let $X$ be any scheme. We will use the short hand $X^s$ for the fiber product $X\times_k\cdots\times_kX$ ($s$ times). Using the uniformization morphism and some facts about arc spaces, we arrive at the following analogue of Theorem \ref{unitheorem} for the multibranched case.

\begin{theorem} \label{multitheorem}
Let $(C,p)$ be a pointed  curve (not necessarily unibranched) satisfying Assumption \ref{ap}. Let $X$ be a $k$-variety, let $r_i=\mbox{mult}_{q_i}(\bar C)$ where $q_i$ are the branches of $C$ at $p$, and let $F_n^i$ be the auxiliary fat point of the branch $q_i$ at level $n$.  Then, there exists natural numbers $e_1,\ldots, e_n$ such that
\begin{itemize}
\item[(1)] The generalized arc space $\prod_{i=1}^{s}\nabla_{F_n^i}X$ is a closed subscheme of $\prod_{i=1}^{s}\nabla_{\fl_{r_in+e_i}}X$.
\item[(2)] The natural truncation morphism $\pi_{n-1}^n: \prod_{i=1}^{s}\nabla_{F_n^i}X\to \prod_{i=1}^{s}\nabla_{F_{n-1}^i}X$ is induced by restricting the natural truncation morphism $\prod_{i=1}^{s}\nabla_{\fl_{r_in+e_i}}X\to \prod_{i=1}^{s}\nabla_{\fl_{r_i(n-1)+e_i}}X$.
\item[(3)] The projective limit $\nabla_{J_p^{\infty}C}X\cong\varprojlim_n \prod_{i=1}^{s}\nabla_{F_n^i}X$ is a closed cylinder of $\nabla_{\fl}X^s$.
\item[(4)] The motivic generating function $Z_S^{r,e}(t)$ defined in Lemma \ref{rzeta} with $S:=\nabla_{J_p^{\infty}C}X$ is well-defined where $e=e_1+\cdots+e_s$.
\item[(5)] If $X$ has a resolution of singularities (e.g., $\mbox{char}(k) = 0$), then $Z_S^{r,e}(t)$ is rational.
\end{itemize}
\end{theorem}

\begin{proof}
Let $C$ be a curve on a smooth surface and let $p$ be a point on $C$. Let $s$ be the cardinality of $\gamma^{-1}(p)$ where $\gamma: \bar C\to C$ is the normalization. Consider the uniformization homomorphism $\hat\sO_{C,p} \into \sqcup_{i=1}^{s}k[[t]]$ defined by $\phi=\sqcup_{i=1}^{s}\phi_i$ with $\phi_i(x) = t^{r_i}$ and $\phi_i(y) = \sum_{j=1}a_{j,i}t^j$. $F_n^i$ is defined to be the spectrum of the ring obtained by modding out $\hat\sO_{C,p}$ with the kernel of the composition of $\phi$ and $\sqcup_{i=1}^{s}k[[t]]\to k[t]/(t^{r_in+e_i})$. 
This induces the closed embedding of $(1)$ and $(2)$ quickly follows as well. The proofs of these facts are essentially the same as the proofs of parts $(1)$ and $(2)$ of Theorem \ref{unitheorem}.

For $(3)$, note that 
$$\nabla_{\fl}X^s =\nabla_{\fl}(\prod_{i=1}^sX) \cong \prod_{i=1}^{s}\nabla_{\fl}X.$$
 Part $(3)$ follows immediately from the fact that $\varprojlim_n\nabla_{\fl_{r_in+e_i}}X \cong \nabla_{\fl}X$, the proof of Theorem \ref{unitheorem}, and the fact that a finite fiber product of closed cylinders is a closed cylinder. The proof for parts $(4)$ and $(5)$ follow from $(3)$ in exactly the same way as in the proof of Theorem \ref{unitheorem}. 
  
\end{proof}

\subsection{Rationality Results.}
Let $C$ be any curve and let $p$ be a point on $C$. The Hilbert-Samuel function $\ell(\sO_{C,p}/\fm_p^n)$ is given by a linear polynomial $P(n) = e_0(C,p)n+e_1(C,p)$ for sufficiently large $n$, where $e_i(C,p)\in\bZ$. The polynomial $P(n)$ is sometimes called the Hilbert-Samuel polynomial of $C$ at $p$, and the number $e_0(C,p)$ is called the Hilbert-Samuel multiplicity of $C$ at $p$. It is well-known (e.g., see the beginning of Chapter $4$, Section $3$ of \cite{WF}) that $e_0(C,p)$ corresponds with the usual geometric notion of multiplicity used above--i.e., we have that $e_0(C,p)=r$ where $r$ is the number used in the work above. This immediately implies the following:

\begin{corollary}\label{rr}
Let 
$(C,p)$ a pointed curve satisfying Assumption \ref{ap}. Assume further that $p$ is the only singular point of $C$.  
Then, there exists $e\in\bN$ and $b\in\bZ$ such that 
 $\Theta_{X,J_p^{\infty}C}(t) -\bL^{b}Z_{S}^{r,e}(t)$ is a polynomial, where $S=\nabla_{J_p^{\infty}C}C$. Moreover, $\Theta_{C,J_p^{\infty}C}(t)$ is a rational function.
 \end{corollary}
 
 \begin{proof}
 For simplicity, we prove the case where $C$ is unibranched. We let $e=(r-1)(m-r+1)$ as in the beginning of \S $4.2$, and we let $v$ be the natural number given by Equation \ref{v}. We let $b= e+v-e_1(C,p)$.
 
 By Lemma \ref{rzeta}, $\bL^{b}Z_{S}^{r,e}(t) = \sum_{n=0}^{\infty}[\pi_{r(n+1)}(S)]\bL^{v-r(n+1) - e_1(C,p)}t^n.$ By Hilbert-Samuel theory as outlined above and by the proof of part $(3)$ of Theorem \ref{unitheorem} and Lemma \ref{vimp}, the right hand side is equal to $p(t)+\sum_{n=q}^{\infty}[\nabla_{J_p^{n+1}C}C]\bL^{-\ell(J_p^{n+1}C)}t^n$ where $p(t)$ is an element of $\grot{\Var{k}}[\bL^{-1},t]$. This proves that $\Theta_{X,J_p^{\infty}C}(t)$ is a rational function by Part $(5)$ of Theorem \ref{unitheorem}. For the multibranch case, we choose $e_i=(r_i-1)(m_i-r_i+1)$ and $v_i$ the number given by Equation \ref{v} for each branch $q_i$. The proof of the result is exactly the same by letting $$b=e_1+\cdots+e_s+v_1+\cdots+v_s -e_1(C,p)$$ provided we apply Theorem \ref{multitheorem} in place of Theorem \ref{unitheorem}.
 \end{proof}
 
 \begin{corollary}\label{ars}
 Let 
$(C,p)$ a pointed curve satisfying Assumption \ref{ap}.
Then, there exists $e\in\bN$ and $b\in\bZ$ such that $\zeta_{C,p}(t) -\bL^{b}Z_{S_a}^{r,e}(t)$ is a polynomial (i.e., elements of $\grot{\Var{k}}[\bL^{-1},t]$), where  $S_a:=\sA(C,p)$. In particular, $\zeta_{C,p}(t)$ is a rational function.
\end{corollary}

\begin{proof}
Follows mutatis mutandis from the proof of Corollary \ref{rr}. Rationality of also follows directly from Theorem \ref{deco} and Corollary \ref{rr}.
\end{proof}
 
\subsection{Revisiting the conjecture of \S 1 } \label{proof} We may now quickly make part of Conjecture \ref{conj1} into a theorem:

\begin{theorem}
Let $C$ be a curve on a smooth surface $S$ over an algebraically closed field $k$ of characteristic zero. Then, $C$ is smooth at $p$ if and only if
$$\zeta_{C,p}(t) = \bL^{-1}\frac{1}{1-t}\ .$$
\end{theorem}

\begin{proof} The forward direction is an immediate corollary to Proposition \ref{first}.
The other direction is now (due to the rationality results in this section) also immediate. Indeed, in the case that $C$ is singular at $p$, it is impossible that $$\zeta_{C,p}(t)= \bL^{-1}\frac{1}{1-t}$$ for the mere reason that $(1-t)Z_{S_a}^{r,e}(t)$ will have a pole (in $t=\bL^{-q}$) by Lemma \ref{vimp}, Corollary \ref{ars}, and the Motivic Monodromy Conjecture, which was proved by Loeser (cf. \cite{L}) in the case of plane curves. Here, $q$ denotes a formal variable over $\bQ$ and $S_a:=\sA(C,p)$.
\end{proof}

\bibliographystyle{amsalpha}
\nocite{*}
\bibliography{auto2}

\providecommand{\bysame}{\leavevmode\hbox to3em{\hrulefill}\thinspace}
\providecommand{\MR}{\relax\ifhmode\unskip\space\fi MR }
\providecommand{\MRhref}[2]{%
  \href{http://www.ams.org/mathscinet-getitem?mr=#1}{#2}
}
\providecommand{\href}[2]{#2}
\begin{thebibliography}{CNS11}

\bibitem[Bli11]{MB}
M.~Blickle, \emph{A short course on geometic motivic integration.},
  pp.~199--243, vol.~1 of London Mathematical Society Lecture Note Series
  \cite{motintbookv1}, 2011.

\bibitem[CA00]{EA}
E.~Casas-Alvero, \emph{Singularities of plane curves}, Lecture note series //
  London Mathematical Society, Cambridge University Press, 2000.

\bibitem[CL08]{CL1}
R.~Cluckers and F.~Loeser, \emph{{Constructible motivic functions and motivic
  integration}}, Invent. Math. \textbf{173} (2008), no.~1, 23--121.

\bibitem[CL15]{CL2}
\bysame, \emph{{Motivic integration in all residue field characteristics for
  Henselian discretely valued fields of characteristic zero}}, J. Reine Angew.
  Math. \textbf{701} (2015), 1--31.

\bibitem[CNS11]{motintbookv1}
R.~Cluckers, J.~Nicaise, and J.~Sebag, \emph{Motivic integration and its
  interactions with model theory and non-archimedean geometry:}, London
  Mathematical Society Lecture Note Series, no. v. 1, Cambridge University
  Press, 2011.

\bibitem[DF98]{DL1}
J.~Denef and Loeser F., \emph{{Motivic Igusa Zeta Functions}}, J. Algebraic
  Geometry \textbf{7} (1998), no.~3, 505--537.

\bibitem[DL99]{DL2}
J.~Denef and F.~Loeser, \emph{{Germs of arcs of singular varieties and motivic
  integration}}, Inven. Math. \textbf{135} (1999), 201--232.

\bibitem[DL02]{DL0}
\bysame, \emph{{Motivic Integration, Quotient Singularities, and McKay
  Correspondence}}, Compositio Math. \textbf{130} (2002), 267--290.

\bibitem[Ful84]{WF}
W.~Fulton, \emph{Intersection theory}, 2 ed., Ergebnisse der Mathematik und
  ihrer Grenzgebiete, Springer-Verlag, 1984.

\bibitem[Har77]{Ha1}
R.~Hartshorne, \emph{Algebraic geometry}, Graduate Texts in Mathematics,
  Springer, 1977.

\bibitem[Kol09]{JK}
J.~Koll{\'a}r, \emph{Lectures on resolution of singularities (am-166)}, Annals
  of Mathematics Studies, Princeton University Press, 2009.

\bibitem[Kon95]{K}
M.~Kontsevich, \emph{{Lecture at Orsay}}, 1995.

\bibitem[Loe88]{L}
F.~Loeser, \emph{{Fonctions d'Igusa p-adiques et polyn\^{o}mes de Bernstein}},
  Amer. J. Math. \textbf{135} (1988), 1--22.

\bibitem[Loe15]{L2}
\bysame, \emph{{personal correspondence}}, 2015.

\bibitem[Sch]{Sch2}
H.~Schoutens, \emph{{Schemic Grothendieck Rings II}},
  \url{websupport1.citytech.cuny.edu/faculty/hschoutens/PDF/SchemicGrothendieckRingPartII.pdf},
  accessed on 10/01/2015.

\bibitem[Sch99]{Sch1}
\bysame, \emph{{Existentially closed models of the theory of Artinian local
  rings}}, J. Symb. Logic \textbf{64} (1999), 825--845.

\bibitem[Sto17]{St}
A.~Stout, \emph{On the auto igusa-zeta function of an algebraic curve}, Journal
  of Symbolic Computation \textbf{79} (2017), 156 -- 185, SI: MEGA 2015.

\end{thebibliography}

\end{document}